\documentclass[12pt]{article}
\usepackage{mathtools, enumerate,amsthm,amssymb}
\usepackage{hyperref}

\newtheorem{thm}{Theorem}
    
    \newtheorem{lmm}[thm]{Lemma}
    \newtheorem{prp}[thm]{Proposition}
    \newtheoremstyle{def} 
    {\topsep}                    
    {\topsep}                    
    {}                   
    {}                           
    {\bfseries}                   
    {.}                          
    {.5em}                       
    {}  
\theoremstyle{def}
\newtheorem{rmrk}[thm]{Remark}

\newcommand{\sm}{C^{\infty}}
\newcommand{\RR}{\mathbb{R}}
\newcommand{\oper}[1]{\operatorname{#1}}
\newcommand{\inv}{^{-1}}
\newcommand{\vm}{\mathfrak{X}}
\newcommand{\vl}{^\mathsf{v}}
\newcommand{\tl}{^\mathsf{c}}
\newcommand{\JJ}{\mathbf{J}}
\newcommand{\slt}{\mathring{T}}
\newcommand{\hll}[1]{\widetilde{#1}}
\newcommand{\ir}[1]{\mathcal{#1}}
\newcommand{\en}{_{i=1}^n}

\title{Isometries, submetries and distance coordinates on Finsler manifolds}
\author{Bernadett Aradi$^*$ and D\'avid Cs.\ Kert\'esz\footnote{The first author's research was supported by the Hungarian Academy of Sciences.\newline Both authors were supported by the T\'AMOP-4.2.2/B-10/1-2010-0024 project and the T\'AMOP 4.2.4. A/2-11-1-2012-0001 ``National Excellence Program -- Elaborating and operating an inland student and researcher personal support system convergence program''. The project is co-financed by the European Union and the European Social Fund.}}
\date{}

\begin{document}
\maketitle

\begin{center}
  {\it Dedicated to Professor Lajos Tam\'assy on the\\ occasion of his 90th birthday}
\end{center}
\begin{abstract}
This paper considers fundamental issues related to Finslerian iso\-metries, submetries, distance and geodesics. It is shown that at each point of a Finsler manifold there is a \emph{distance coordinate system}. Using distance coordinates, a simple proof is given for the Finslerian version of the Myers--Steenrod theorem and for the differentiability of Finslerian submetries.
\end{abstract}

AMS \emph{Subject Class.}\ (2010):
53B40 


\section*{Introduction}
The main actors of this paper are isometries, distance preserving maps and geodesics of Finsler manifolds. By a \emph{Finslerian isometry} we mean a diffeomorphism between Finsler manifolds whose derivative leaves the Finsler function invariant. Every Finsler manifold carries a natural \emph{quasi-metric}, which is called the \emph{Finslerian distance}. Thus we can measure the distance of points in (connected) Finsler manifolds, and we may speak of \emph{distance preserving mappings} between them. The \emph{geodesics} of a Finsler manifold are constant speed curves that are extremals of the arc length functional. These curves locally minimize the Finslerian distance, and they can also be described by the canonical spray of the Finsler manifold.

It sounds natural that a Finslerian isometry preserves the geodesics of the Finsler manifold, however, we were having trouble finding a proof in the literature. In fact the only one we found was in the unpublished manuscript \cite{LibMo}, and it is a lengthy local calculation. For the very special case of Berwald manifolds, S.~Deng proved the proposition in his book \cite{Deng12:hom}. We provide a new proof for the general case using a characterization of the canonical spray (Proposition \ref{isoaff}). As an application of this result, we show that a Finslerian isometry is completely determined by the action of its derivative on one tangent space only (Proposition \ref{isodet}).

It is quite clear that a Finslerian isometry preserves the Finslerian distance. The converse is the Finslerian version of the Myers--Steenrod theorem, which states that a surjective distance preserving mapping between Finsler manifolds is a Finslerian isometry.
The result was first proved by F.~Brickell in 1965, and again, by S.~Deng and Z.~Hou in 2002 (see \cite{Bric65:diff} and \cite{DeHo02:isom}). However, both proofs are quite technical. For the Riemannian version of the theorem an elegant proof can be found in Petersen's book \cite{Pet00:Riem}. His argument uses special coordinates called \emph{distance coordinates}, where the coordinates of a point are the point's distances from some given points. We carried out this idea to the Finslerian setting. The difficult part was to show that distance coordinate systems in a Finsler manifold exist. Then the smoothness of distance preserving maps follows easily, and this significantly simplifies the proof in \cite{Bric65:diff} and \cite{DeHo02:isom}.

As a further application of distance coordinates, we show that regular submetries between reversible Finsler manifolds are differentiable (Theorem \ref{thm:Beres}). This result is known in the more general setting of metric spaces (see \cite{Lyt}). We believe, however, that our proof is more accessible to researchers working in differential geometry.

The paper is organized as follows. In the first section we collect our notations and conventions about manifolds and quasi-metric spaces. Next we discuss sprays and the exponential map determined by a spray, these are essential tools for our later arguments. The remaining sections are devoted to proving our results.

\section{Preliminaries}
\noindent{\bf(1)} \ Throughout the paper $M$ denotes a connected, second countable, smooth Hausdorff manifold with finite dimension $n$. The tangent bundle of $M$ is denoted by $\tau\colon TM\to M$, and we use the notation $\slt M$ for the bundle of nonzero tangent vectors to $M$. The tangent bundle of the tangent manifold $TM$ is $\tau_{TM}\colon TTM\to TM$.

We denote by $\sm(M)$ the real algebra of smooth functions on $M$, and by $\vm(M)$ the $\sm(M)$-module of vector fields on $M$. If $M,N$ are smooth manifolds and $\varphi\colon M\to N$ is a smooth mapping, then its derivative is $\varphi_*\colon TM\to TN$. Two vector fields $X\in \vm(M)$ and $Y\in \vm(N)$ are $\varphi$-related if $\varphi_*\circ X=Y \circ \varphi$. The vertical and the complete lifts of $X$ into $TM$ are denoted by $X\vl$ and $X\tl$, respectively. There exists a unique $(1,1)$ tensor field $\JJ$ on $TM$ such that $\JJ X\vl=0$ and $\JJ X\tl=X\vl$ for any vector field $X\in \vm(M)$. This tensor is called the vertical endomorphism of $\vm(TM)$. For the canonical radial vector field (or Liouville vector field) on $TM$ we use the notation $C$; it is generated by the positive dilations of $TM$.

The exterior derivative is denoted by $d$, and $i_X$ stands for the substitution operator associated to a vector field $X\in \vm(M)$. It acts on a tensor field $A$ of type $(0,k)$ or $(1,k)$ by $i_XA(X_1,\dots,X_{k-1}):=A(X,X_1,\dots,X_{k-1})$.\bigskip

\noindent{\bf(2)} \ Let $H$ be an arbitrary set. A function $\varrho\colon H\times H\to \RR,\ (p,q)\mapsto \varrho(p,q)$ is a \emph{quasi-distance} on $H$ if it is non-negative, $\varrho(p,q)=0$ if and only if $p=q$, satisfies the triangle inequality, and the \emph{forward metric balls}
\[
B_\varrho^+(a,r):=\{p\in H\mid \varrho(a,p)<r\}
\]
generate the same topology as the \emph{backward metric balls}
\[
B_\varrho^-(a,r):=\{p\in H\mid \varrho(p,a)<r\}
\]
where $a\in H$ and $r$ is a positive number. In this case $(H,\varrho)$ is said to be a \emph{quasi-metric space} \cite{Pat05}. Later on, we need only forward metric balls denoted simply by $B_\varrho(a,r)$. Furthermore, let $S_\varrho(a,r):=\{p\in H\mid \varrho(a,p)=r\}$ denote the sphere around $a$ with radius $r$.

Consider two quasi-metric spaces $(H_1,\varrho_1)$ and $(H_2,\varrho_2)$. A mapping $\varphi$ from $H_1$ to $H_2$ is \emph{distance preserving}, if
\[
\varrho_1(p,q)=\varrho_2(\varphi(p),\varphi(q));\quad p,q\in H_1.
\]
We say that $\varphi\colon H_1\to H_2$ is a \emph{submetry}, if it satisfies the following weaker condition: for any $p$ in $H_1$, there is a positive number $\delta$ such that for every $\varepsilon\in\left]0,\delta\right[$ we have
\[
\varphi(B_{\varrho_1}(p,\varepsilon))=B_{\varrho_2}(\varphi(p),\varepsilon).
\]
For each $p\in H_1$, the supremum of these positive numbers $\delta$ will be denoted by $\delta_p$. We say that a submetry is \emph{regular} if for any compact set $K\subset H_1$ we have
\[
\delta_K:=\inf_{p\in K}\delta_p>0.
\]
The following properties of submetries can be immediately deduced from the definition:
\begin{enumerate}[(i)]
\itemsep-3pt
\item submetries are continuous;
\item composition of submetries is a submetry;
\item composition of regular submetries is a regular submetry.
\end{enumerate}

\section{Spray manifolds and the exponential map}
A mapping $S\colon TM\to TTM$ is said to be a \emph{spray} for $M$ if it is a section of the double tangent bundle $\tau_{TM}\colon TTM\to TM$, smooth on $\slt M$, positive-homogeneous of degree $2$ and satisfies $\JJ S=C$. In this case the pair $(M,S)$ is called a \emph{spray manifold}.

By a \emph{geodesic} of a spray $S$ we mean a smooth curve $\gamma\colon I\to M$ whose velocity field is an integral curve of $S$, that is, $S\circ \dot{\gamma}=\ddot{\gamma}$. Given a vector $v$ in $TM$, there exists a unique maximal geodesic $\gamma_v\colon I\to M$ such that $0\in I$ and $\dot\gamma_v(0)=v$. The $2^+$-homogeneity of $S$ implies that if $s$ and $t$ are positive numbers such that $\gamma_v$ is defined at $st$, then $\gamma_{t v}$ is defined at $s$ and
\begin{equation}\label{eq:rescale}
\gamma_{t v}(s)=\gamma_v(s t).
\end{equation}

Let $\hll{TM}$ be the set of tangent vectors $v\in TM$ such that $\gamma_v$ is defined at $1$. Then the \emph{exponential map} determined by $S$ is the mapping
\[
\exp\colon\hll{TM}\to M,\quad v\mapsto\exp(v):=\gamma_v(1).
\]
The set $\hll{TM}$ is open in $TM$ \cite[p.~91]{Lang95:diff}, and, by the smooth dependence on initial conditions, $\exp$ is smooth on $\hll{TM}\cap\slt M$. Applying \eqref{eq:rescale}, we see that the curve $t\in[0,1]\mapsto \exp(tv)$ is a geodesic of $S$ with initial velocity $v$.

The restriction of the exponential map to the tangent space at a certain point $p$ is denoted by $\exp_p$. The following result is well-known, its proof can be found, e.g., in \cite[p.~222]{Shen01:Spra}.
\begin{lmm}\label{thm:Wh1}
If $\exp$ is the exponential map determined by a spray, then for any $p\in M$,
\begin{itemize}
\itemsep-3pt
\item[\textup{(i)}] $\exp_p$ is of class $C^1$ on $T_pM\cap \hll{TM}$;
\item[\textup{(ii)}] $((\exp_p)_*)_{0_p}$ is the canonical isomorphism which identifies $T_{0_p}T_pM$ with $T_pM$.
\end{itemize}
\end{lmm}
From this result, it follows that there is a neighbourhood of $0_p\in T_pM$ on which $\exp_p$ is a $C^1$ diffeomorphism onto its image.
Suppose that we have such a neighbourhood which is also star-shaped with respect to the origin. Then its image under $\exp_p$ is called a \emph{normal neighbourhood} of $p$. A normal neighbourhood $\ir U$ of $p$ has the nice property that each of its points lies on a geodesic starting from $p$, contained in $\ir U$, and all such geodesics differ only in a reparametrization.

An open subset of a spray manifold is called \emph{totally normal}, if it is a normal neighbourhood of each of its points. A fundamental result of J.~H.~C.~Whitehead \cite{Whit33:totnor} assures that in a spray manifold any point has a totally normal neighbourhood. It is worth mentioning that later R.~E.~Traber \cite{Trab37:fund} gave a simpler proof for this assertion.
\bigskip

The following lemma will be useful at the construction of distance coordinates in Section 4.

\begin{lmm}\label{lmm:assocp}
Let $(M,S)$ be a spray manifold, $p$ a point in $M$. Then for any nonzero $v\in T_pM$, there is a point $q\neq p$ in $M$ and a nonzero $w\in T_qM$ satisfying the following conditions:
\begin{enumerate}
\itemsep-3pt
\item[\textup{(i)}] $p$ is in a normal neighbourhood of $q$;
\item[\textup{(ii)}] the geodesic $\gamma_w\colon t\mapsto\exp_q(tw)$ runs through $p$, and its velocity at $p$ is $\lambda v$ for some $\lambda>0$.
\end{enumerate}
\end{lmm}

\begin{proof}
Let $v\in T_pM$ be a nonzero tangent vector and consider the geodesic $\gamma_v$ with velocity $v$ at $0$. It is defined on an open interval containing $[-\delta,\delta]$ for some $\delta>0$. Since $\gamma_v$ is continuous, we may suppose that $\gamma_v([-\delta,\delta])$ is contained in a totally normal neighbourhood $\ir U$ of $p$. So if we set $q:=\gamma_v(-\delta)$ and take into account that $\ir U$ is a normal neighbourhood of any of its points, we obtain (i).

To prove (ii), let $\ir V$ be an open subset in $T_qM$ such that $\exp_q\upharpoonright \ir V$ is a diffeomorphism onto $\ir U$ and let $w_0:=\dot\gamma_v(-\delta)$. There exists a positive number $\lambda$ such that $w:=\lambda w_0\in(\exp_q\upharpoonright\ir V)\inv(\ir U)\subset T_qM$. Then the geodesic
\[
\gamma_w\colon t\mapsto \gamma_w(t)=\gamma_{tw}(1)=\exp_q(tw)
\]
is a positive affine reparametrization of $\gamma_v$ expressed by $\gamma_w(t)=\gamma_v(\lambda t-\delta)$. Indeed, both $\gamma_w$ and $t\mapsto \gamma_v(\lambda t-\delta)$ are geodesics with common initial velocity $\dot\gamma_w(0)=w=\lambda \dot\gamma_v(-\delta)$, thus they must coincide. Hence
\[
\gamma_w\left(\frac\delta\lambda\right)=\gamma_v(0)=p \ \mbox{ and } \ \dot\gamma_w\left(\frac\delta\lambda\right)=\lambda\dot\gamma_v(0)=\lambda v,
\]
as it was to be shown.
\end{proof}

In the situation described by the lemma we say that $q$ is an \emph{emanating point} of the vector $v$. Given a vector $v\in T_pM$, from the construction above we see that any neighbourhood of $p$ contains an emanating point of $v$.

\section{Finsler functions and their isometries}
A continuous function $F\colon TM\to \RR$ is called a \emph{Finsler function} (for $M$) if it is smooth on $\slt M$, positive-homogeneous of degree $1$, positive on the nonzero tangent vectors and \emph{elliptic}, that is, if $F_p:=F\upharpoonright T_pM$, then the symmetric bilinear form $(F_p^2)''(v)$ is positive definite for all $p\in M$ and $v\in \slt_pM$. In this case we say that the pair $(M,F)$ is a \emph{Finsler manifold}. The function $E=\frac{1}{2}F^2$ is said to be the \emph{energy function} of $(M,F)$ or simply of $F$. A Finsler function $F$ is \emph{reversible} if $F(-v)=F(v)$ for all $v\in \slt M$.

Let $(M,F)$ be a Finsler manifold. There exists a unique spray $S$ on $TM$ such that $i_Sd(dE\circ \JJ)=-dE$. This spray is called the \emph{canonical spray} of $(M,F)$. A smooth curve $\gamma\colon I\to M$ is called a \emph{geodesic} of the Finsler manifold if it is a geodesic of its canonical spray. By the \emph{exponential map} determined by a Finsler function we mean the exponential map determined by its canonical spray.

The following lemma characterizes the canonical spray of a Finsler manifold.

\begin{lmm}\label{lmm:canspr}
  Let $(M,F)$ be a Finsler manifold. For a spray $S\colon TM\to TTM$ the following conditions are equivalent:
  \begin{itemize}
  \itemsep-3pt
    \item[\textup{(i)}] $S$ is the canonical spray of $(M,F)$;
    \item[\textup{(ii)}] $S(X\vl E)-X\tl E=0$ for all $X\in \vm(M)$;
    \item[\textup{(iii)}] $SF=0$ and $S(X\vl F)-X\tl F=0$ for all $X\in \vm(M)$.
  \end{itemize}
\end{lmm}

\begin{proof}
  The equivalence of (i) and (ii) follows from the proof of Fact 3 in \cite{SzLK11}.
  
  We continue by showing that (i) (and hence (ii)) implies $SF=0$. Indeed, since $d(dE\circ\JJ)$ is an alternating form, we have 
  \[
  0=d(dE\circ\JJ)(S,S)=(i_Sd(dE\circ\JJ))(S)=-dE(S)=-SE=-F(SF).
  \]
  Finally, by the computation below it follows that (ii) and (iii) are equivalent:
  \begin{align*}
    S(X\vl E)-X\tl E&=S(FX\vl F)-FX\tl F\\&=(SF)(X\vl F)+F(S(X\vl F)-X\tl F).
  \end{align*}
\end{proof}

In the previous lemma one of the \emph{Rapcs\'ak equations} \cite{SzV} appears: condition (iii) states that the canonical spray is characterized by a Rapcs\'ak equation and by $SF=0$. This is not surprising, since the Rapcs\'ak equations express that a given spray is projectively equivalent to the canonical spray of a Finsler manifold.

Let $(M,F)$ and $(N,\bar F)$ be Finsler manifolds. A smooth mapping $\varphi$ of $M$ onto $N$ is said to be a \emph{Finslerian isometry} if it is a diffeomorphism and its derivative preserves the Finsler norms of tangent vectors, i.e., $\bar F\circ \varphi_*=F$. The mapping $\varphi\colon M\to N$ is called a \emph{local Finslerian isometry} if every point has a neighbourhood on which $\varphi$ is a Finslerian isometry.
\bigskip

Now we show that a Finslerian isometry preserves the geodesics of the Finsler manifold. The Riemannian version of the proposition is well-known, but, for example, in his book Deng proved the result (as well as its corollary Proposition \ref{isodet} below) for the special case of Berwald manifolds only \cite[Theorem 5.1. and Theorem 5.2.]{Deng12:hom}. In our proof we apply the characterization of the canonical spray stated in Lemma \ref{lmm:canspr}.

\begin{prp}\label{isoaff}
Let $(M,F)$ be a Finsler manifold, $\gamma\colon I \to M$ a geodesic of $F$ and $\varphi\colon M \to M$ a Finslerian isometry. Then $\varphi \circ \gamma$ is a geodesic as well, that is, isometries preserve geodesics.
\end{prp}

\begin{proof}
Let $S$ denote the canonical spray of $(M,F)$, and let $\hll{\gamma}:=\varphi \circ \gamma$, for brevity. Then we have to show that $S\circ \dot{\gamma}=\ddot{\gamma}$ implies $S \circ \dot{\hll{\gamma}}=\ddot{\hll{\gamma}}$. We prove that
\begin{equation}\label{spr}
\varphi_{**}\circ S=S\circ \varphi_*,
\end{equation}
because then we obtain
\[
\ddot{\hll{\gamma}}=\varphi_{**}\circ \ddot{\gamma}\overset{\textrm{cond.}}{=}\varphi_{**}\circ S\circ \dot{\gamma}\overset{\eqref{spr}}{=}S\circ \varphi_* \circ \dot{\gamma}
=S\circ \dot{\hll{\gamma}},
\]
which completes the proof.

To prove \eqref{spr}, we introduce a mapping $\hll{S}\colon TM\to TTM$ defined by
\[
\hll{S}:=\varphi_{**}\circ S\circ \varphi_*\inv,
\]
and show that it is the canonical spray of $F$. Then, by the uniqueness of the canonical spray, $S=\hll{S}$ and hence \eqref{spr} follows.

The mapping $\hll{S}\colon TM\to TTM$ is clearly a spray, since it is the push-forward of a spray by a diffeomorphism. For a fixed $X\in \vm(M)$, define the vector field $Y:=\varphi_*\inv\circ X\circ \varphi\in \vm(M)$. Then $Y\tl=\varphi_{**}\inv \circ X\tl \circ \varphi_*$ and $Y\vl=\varphi_{**}\inv \circ X\vl \circ \varphi_*$, (see \cite[Section 5]{PSz}), which mean that $Y\tl$ and $X\tl$ are $\varphi_*$-related, as well as $Y\vl$ and $X\vl$. Thus, by the well-known characterization of $\varphi_*$-relatedness (see, e.g., \cite[p.~109, Lemma~5]{GHV1}),
\begin{align}\label{teljes}
Y\tl E \circ \varphi_*\inv=Y\tl(E\circ \varphi_*) \circ\varphi_*\inv=(X\tl E\circ\varphi_*)\circ\varphi_*\inv=X\tl E,
\end{align}
where at the first step we used the fact that $\varphi$ is a Finslerian isometry and hence $E\circ \varphi_*=E$. We obtain $X\vl E \circ \varphi_*=Y\vl(E\circ \varphi_*)=Y\vl E$ with a similar computation, which yields
\begin{equation}\label{vert}
\begin{aligned}
&\hll{S}(X\vl E)=(\varphi_{**}\circ S\circ \varphi_*\inv)(X\vl E)\\
&\quad=(S\circ \varphi_*\inv)(X\vl E\circ \varphi_*)=(S\circ \varphi_*\inv)(Y\vl E).
\end{aligned}
\end{equation}
Finally,
\begin{align*}
X\tl E-\hll{S}(X\vl E)&\overset{\eqref{teljes},\eqref{vert}}{=}Y\tl E\circ\varphi_*\inv-(S\circ \varphi_*\inv)(Y\vl E)\\
&\hspace{7.5pt}=(Y\tl E-S(Y\vl E))\circ \varphi_*\inv\overset{\textrm{Lemma \ref{lmm:canspr}}}{=}0,
\end{align*}
since $S$ is the canonical spray of $F$. This implies, by Lemma \ref{lmm:canspr} again, that $\hll{S}$ is a canonical spray of $F$ as well, as wanted.
\end{proof}

\begin{rmrk}\label{lociso}
The previous proposition remains true if we consider a \emph{local} Finslerian isometry. Furthermore, the result also holds if $\varphi$ is a (local) Finslerian isometry between two (possibly) different Finsler manifolds.
\end{rmrk}

As a simple application of Proposition \ref{isoaff} we show that a local Finslerian isometry is determined by the action of its derivative on a single tangent space. For the Riemannian case, see, e.g., \cite[p.~91]{ONeill}.

\begin{prp}\label{isodet}
  Let $\varphi,\psi\colon M\to M$ be local isometries of the (connected) Finsler manifold $(M,F)$. If there exists a point $p\in M$ such that $(\varphi_*)_p=(\psi_*)_p$, then $\varphi=\psi$ on $M$.
\end{prp}

\begin{proof}
  Let $A:=\{q\in M \mid (\varphi_*)_q=(\psi_*)_q\}$. Then $A$ is nonempty by assumption and also closed by the continuity of $\varphi_*$ and $\psi_*$. We show that $A$ is open, whence $A=M$, and we are done.

  Let $q\in A$ be fixed and let $\ir U_q$ be a normal neighbourhood of $q$. If $r$ is another point in $\ir U_q$ then there exists a tangent vector $v\in T_qM$ such that $r=\exp_q(v)=\gamma_v(1)$, where $\gamma_v$ denotes the unique maximal geodesic of $(M,F)$ with initial velocity $v$. Since a local isometry $\varphi\colon M\to M$ sends geodesics into geodesics (see Proposition~\ref{isoaff} and Remark~\ref{lociso}), we have that $\varphi\circ\gamma_v$ is a geodesic as well with initial velocity $\varphi_*(\dot{\gamma}_v(0))=\varphi_*(v)\in T_{\varphi(q)}M$. Naturally, we denote this unique geodesic by $\gamma_{\varphi_*(v)}$.

  The same assertions hold for the local isometry $\psi$, that is, $\gamma_{\psi_*(v)}:=\psi\circ\gamma_v$ is a geodesic of $(M,F)$. Thus we obtain
  \[
  \varphi(r)=\varphi(\gamma_v(1))=\gamma_{\varphi_*(v)}(1)\overset{q\in A}{=}\gamma_{\psi_*(v)}(1)=\psi(\gamma_v(1))=\psi(r)
  \]
  for all $r\in \ir U_q$. Therefore, $\varphi_*=\psi_*$ over $T\ir U_q$, and  $\ir U_q\subset A$. So it follows that a normal neighbourhood of an arbitrary point in $A$ is a subset of $A$, whence the openness of $A$.
\end{proof}

\section{Finslerian distance coordinates}
In this section first we introduce an intrinsic distance function on a Finsler manifold and collect its most important properties (for details we refer to \cite{BCS00:intr}). Next we show how to construct a so-called distance coordinate system around a point of a Finsler manifold.
\bigskip

Let $(M,F)$ be a Finsler manifold. The \emph{arc length} of a (piecewise differentiable) curve $\gamma\colon[a,b]\to M$ is
\[
L(\gamma):=\int_a^bF\circ\dot\gamma.
\]
Constant speed curves that are stationary points of this functional are exactly the geodesics of $F$ (see, e.g., \cite{SziZ10:diss}).

In a Finsler manifold $(M,F)$ the distance of two points $p,q\in M$ can be measured as follows. Let $\Gamma(p,q)$ denote the set of piecewise smooth curves $\gamma\colon[0,1]\to M$ such that $\gamma(0)=p$, $\gamma(1)=q$. Then the \emph{Finslerian distance} from $p$ to $q$ is
\[
\varrho(p,q):=\inf\{L(\gamma)\in\RR\mid\gamma\in\Gamma(p,q)\},
\]
and $\varrho\colon M\times M\to\RR$ is a quasi-metric on $M$. In this case the topology generated by the forward (or backward) metric balls is just the underlying manifold topology. Furthermore, if $F$ is reversible then $\varrho$ is symmetric, and hence it becomes a metric on $M$. A mapping between Finsler manifolds is called distance preserving (or a submetry) if it is distance preserving (or a submetry) with respect to the Finslerian distances. For a fixed point $p$ in $M$, the function $r_p:=\varrho(p,\cdot)$ is called the \emph{distance function} from $p$.

Let $(M,F)$ and $(N,\bar F)$ be two Finsler manifolds, and $\varphi\colon M\to N$ a Finslerian isometry. Then $\varphi$ is distance preserving as well, that is, for all $p,q$ in $M$ we have $\varrho_F(p,q)=\varrho_{\bar F}(\varphi(p),\varphi(q))$. Indeed, the Finslerian distances $\varrho_F$ and $\varrho_{\bar F}$ are defined with the help of the arc length of piecewise smooth curves. However, the length of a curve does not change under a Finslerian isometry, since
\[
 F\circ\dot{\overbracket[.4pt]{\varphi\circ\gamma}}=F\circ\varphi_*\circ\dot\gamma=F\circ\dot\gamma.
\]
Thus a Finslerian isometry is a surjective distance preserving map. The converse is also true, but much less trivial. Then the result is the Finslerian version of the well-known Myers--Steenrod theorem, see later Theorem \ref{thm:MyersSt}.

Let $\ir U$ be a normal neighbourhood of $p\in M$ and $\ir V\subset T_pM$ an open subset containing $0_p$ such that $\exp_p\upharpoonright\ir V$ is a diffeomorphism onto $\ir U$. Then for each $q\in\ir U$ we have
\begin{equation}\label{eq:expF}
r_p(q):=\varrho(p,q)=F((\exp_p\upharpoonright\ir V)\inv(q)).
\end{equation}
For a proof, see \cite[Theorem 6.3.1.]{BCS00:intr}.
From this we see that $r_p$ is smooth on a normal neighbourhood of $p$, except at the point $p$. Taking into account \eqref{eq:rescale}, it follows from \eqref{eq:expF} that for any $v\in T_pM$ and positive number $t$ such that $tv\in\ir V$ we have
\begin{equation}\label{eq:geoddist}
\varrho(p,\exp_p(tv))=\varrho(p,\gamma_v(t))=tF(v).
\end{equation}

Relation \eqref{eq:geoddist} makes it possible to reconstruct the Finsler function $F$ from the quasi-metric $\varrho$. The exact (and more general) formula is given by the next result.

\begin{lmm}[Busemann--Mayer theorem \cite{BusMay48}]\label{lmm:Fdist}
  Let $(M,F)$ be a Finsler manifold. Given a vector $v\in TM$ and a smooth curve $\alpha$ in $M$ such that $\dot\alpha(0)=v$, we have
\begin{equation*}
F(v)=\lim_{t\to0^+}\frac1t \varrho(\alpha(0),\alpha(t)).
\end{equation*}
\end{lmm}

Now we turn to the construction of distance coordinates mentioned earlier.

\begin{prp}\label{prp:distcord}
Given a point $p$ in a Finsler manifold, there is a sequence $(p_i)\en$ of points such that $\theta:=(r_{p_1},\dots,r_{p_n})$ is a diffeomorphism from an open neighbourhood of $p$ onto an open subset of $\RR^n$.
\end{prp}

\begin{proof}
Choose a nonzero vector $v_1$ in $T_pM$, and let $p_1$ be an emanating point of $v_1$ (see the very end of Section 2). Then $p\in S_\varrho(p_1,\varrho_1)=:S_1$ for some $\varrho_1>0$. Now choose a nonzero vector $v_2$ in $T_p S_1$ and introduce $p_2$ and $S_2$ analogously. Continue this construction with vectors $v_k\in \cap_{i=1}^{k-1}T_pS_i$ for $k\in\{3,\dots,n\}$. The subspace $\cap_{i=1}^{k-1}T_pS_i$ is nontrivial, since the intersection of $k$ subspaces of dimension $n-1$ has dimension of at least $n-k$ by Sylvester's rank inequality. So we can choose a nonzero vector $v_k$ in every step of this construction.

Now consider the mapping $\theta:=(r_{p_1},\dots,r_{p_n})$ where $r_{p_i}$ is the distance function from $p_i$ for $i\in\{1,\dots,n\}$. Then, by Lemma \ref{lmm:assocp}(i), $\theta$ is smooth on a neighbourhood of $p$. We show that $(\theta_*)_p$ is bijective, hence $\theta$ is a diffeomorphism from an open subset $\ir D$ containing $p$ onto an open subset of $\RR^n$.

Let $(e^j)_{j=1}^n$ be the canonical coordinate system of $\RR^n$. Then, for any fixed $i,j\in\{1,\dots,n\}$, we have $(\theta_*)_p(v_i)(e^j)=v_i(e^j\circ\theta)=v_i(r_{p_j})$. Since $p_i$ is an emanating point of $v_i$, we have $v_i=\lambda_i\dot\gamma_{w_i}(\varrho_i)$ for some $w_i\in T_{p_i}M$ and $\lambda_i>0$. Then
\begin{multline*}
\frac1{\lambda_i}v_i(r_{p_i})=\dot\gamma_{w_i}(\varrho_i)(r_{p_i})=\left(\frac{d}{dt}\right)_{\hspace{-3pt}\varrho_i}(r_{p_i}\circ\gamma_{w_i}) \\=\left(\frac{d}{dt}\right)_{\hspace{-3pt}\varrho_i}(s \mapsto r_{p_i}\circ\exp_{p_i}(s w_i))
\overset{\eqref{eq:expF}}{=}\left(\frac{d}{dt}\right)_{\hspace{-3pt}\varrho_i}(s\mapsto s F(w_i))=F(w_i)\neq0.
\end{multline*}
Furthermore, $v_i(r_{p_j})=0$ if $j<i$, because $v_i\in T_pS_j$. Thus the Jacobian matrix $\big((\theta_*)_p(v_i)(e^j)\big)$ is strictly lower triangular, so $(v_i)\en$ must be a basis of $T_pM$, and hence $(\theta_*)_p$ is bijective.
\end{proof}

The pair $(\ir D,\theta)$ constructed in the proof above is called a \emph{distance coordinate system} at $p$ and the points $(p_i)\en$ are mentioned as its \emph{base points}.

\begin{rmrk}\label{rmrk:closedistcord}
Since emanating points of vectors in $T_pM$ can be chosen arbitrarily close to $p$, for any neighbourhood $\ir U$ of $p$ we can choose a distance coordinate system with base points contained in $\ir U$.
\end{rmrk}

\section{The Finslerian Myers--Steenrod theorem}

In this section we present a simple proof of the Finslerian version of the Myers--Steenrod theorem, using distance coordinates. This result was first obtained by Brickell (see \cite{Bric65:diff}) and was rediscovered by S.~Deng and Z.~Hou \cite{DeHo02:isom}. The idea of our proof is the same as that of Petersen \cite{Pet00:Riem} in the case of Riemannian manifolds.

\begin{thm}\label{thm:MyersSt}
  A surjective distance-preserving map between Finsler manifolds is a Finslerian isometry.
\end{thm}

\begin{proof}
  Let $(M,F)$ and $(N,\bar F)$ be Finsler manifolds with Finslerian distances $\varrho$ and $\bar \varrho$, respectively, and let $\varphi\colon M\to N$ be a surjective distance-preserving mapping. Then, obviously, $\varphi$ is also injective.

  First we prove that $\varphi$ is smooth, and hence it is a diffeomorphism. Fix a point $p\in M$ and let $q:=\varphi(p)$. Forward balls generate the topology of $M$, therefore for a sufficiently small $r>0$, we can assume that $B_\varrho(p,r)$ is contained in a totally normal neighbourhood of $p$. Since $\varphi$ is surjective and preserves distance, we have $B_{\bar \varrho}(q,r)=\varphi(B_\varrho(p,r))$.

  By Remark \ref{rmrk:closedistcord}, there is a distance coordinate system $(\ir D,\theta)$ at $q$ with base points $(q_i)\en$ contained in $B_{\bar \varrho}(q,r)$. Setting $p_i:=\varphi\inv(q_i)$ for $i\in\{1,\dots,n\}$, we have $p_i\in B_\varrho(p,r)$, and for all $a\in\varphi\inv(\ir D)\subset M$ relation
  \[
    r_{p_i}(a)=\varrho(p_i,a)=\bar \varrho(q_i,\varphi(a))=\bar r_{q_i}(\varphi(a))
  \]
  holds, where $r_{p_i}$ and $\bar r_{q_i}$ denote distance functions in $M$ and $N$, respectively. Therefore, $\theta\circ\varphi=(\bar r_{q_1}\circ\varphi,\dots,\bar r_{q_n}\circ\varphi)=(r_{p_1},\dots,r_{p_n})$. Since $B_\varrho(p,r)$ is contained in a totally normal neighbourhood of $p$, the functions $r_{p_i}$ are smooth on $B_\varrho(p,r)\setminus\{p_i\}$ for all $i\in\{1,\dots,n\}$. Furthermore, $\theta$ is a diffeomorphism on a neighbourhood of $q=\varphi(p)$, therefore $\varphi$ is smooth at $p$.

  Secondly, we show that $\bar F\circ \varphi_*=F$. For this purpose, let $v\in TM$, and let $\alpha$ be a smooth curve in $M$ such that $\dot\alpha(0)=v$. Then $\varphi\circ\alpha$ is a smooth curve in $N$ with the property $\dot{\overbracket[.4pt]{\varphi\circ\alpha}}(0)=\varphi_*(v)$. So, applying Lemma \ref{lmm:Fdist}, we obtain
  \[
    F(v)=\lim_{t\to 0^+}\frac 1 t \varrho(\alpha(0),\alpha(t))=\lim_{t\to 0^+}\frac 1 t \bar \varrho\big(\varphi(\alpha(0)),\varphi(\alpha(t))\big)=\bar F(\varphi_*(v)).
  \]
\end{proof}

\section{Finslerian submetries}

To prepare the main result of this concluding section, we start with the following observation.

\begin{lmm}\label{lmm:distsub}
Let $(M,F)$ be a reversible Finsler manifold, $p$ a point in $M$, and\/ $\ir U$ a normal neighbourhood of $p$. Then the distance function $r_p$ restricted to $\ir U\setminus\{p\}$ is a regular submetry into $\RR$.
\end{lmm}

\begin{proof}
Choose a point $q$ in $\ir U\setminus\{p\}$. Let $\delta$ be the minimum of the two numbers $\varrho(p,q)$ and $\varrho(q,M\setminus\ir U):=\inf\{\varrho(q,\hll{q})\mid\hll{q}\in M\setminus\ir U\}$. Fix $\varepsilon\in\left]0,\delta\right[$ and $a\in B_\varrho(q,\varepsilon)$.
The Finslerian distance satisfies the triangle inequality, therefore
\begin{align*}
\varrho(p,a)&\leq \varrho(p,q)+\varrho(q,a),\\
\varrho(p,q)&\leq \varrho(p,a)+\varrho(a,q).
\end{align*}
Rearranging these inequalities, and using that $F$ is reversible (whence $\varrho$ is symmetric), we obtain $|\varrho(p,a)-\varrho(p,q)|\leq \varrho(q,a)<\varepsilon$. Consequently,
\[
r_p(B_\varrho(q,\varepsilon))\subset \left]r_p(q)-\varepsilon,r_p(q)+\varepsilon\right[=:B(r_p(q),\varepsilon).
\]

Now we show that $r_p$ maps $B_\varrho(q,\varepsilon)$ onto $B(r_p(q),\varepsilon)$. Let $\gamma$ be the maximal unit speed geodesic starting from $p$ passing through $q$. For any positive $t$ such that $\gamma([0,t])$ is contained in $\ir U$, we have
\begin{equation}\label{eq:rpgeod}
r_p(\gamma(t))=\varrho(p,\gamma(t))\overset{\eqref{eq:geoddist}}{=}t.
\end{equation}
Since $\varepsilon<r_p(q)$, the interval $B(r_p(q),\varepsilon)$ contains only positive numbers. Furthermore, $B_\varrho(q,\varepsilon)\subset\ir U$ because $\varepsilon<\varrho(q,M\setminus \ir U)$. Thus, according to \eqref{eq:rpgeod}, we only have to show that $\gamma(t)\in B_\varrho(q,\varepsilon)$ if $t\in B(r_p(q),\varepsilon)$.

First assume that $t\in\left[r_p(q),r_p(q)+\varepsilon\right[$. The curve $\gamma$ is of unit speed, so for any $t\in\left[r_p(q),r_p(q)+\varepsilon\right[$, the length of the curve segment $\gamma\upharpoonright[r_p(q),t]$ is equal to $t-r_p(q)$. Notice that \eqref{eq:rpgeod} implies $\gamma(r_p(q))=q$. Then, by the definition of the Finslerian distance, we have
\[
\varrho(q,\gamma(t))=\varrho(\gamma(r_p(q)),\gamma(t))\leq t-r_p(q)<\varepsilon.
\]
Secondly, suppose that $t\in\left]r_p(q)-\varepsilon,r_p(q)\right]$. Using the reversibility of $F$, we obtain similarly that $\varrho(q,\gamma(t))\leq r_p(q)-t<\varepsilon$, as was to be shown.

The regularity of $r_p$ follows from the fact that disjoint closed and compact sets have positive distance.
\end{proof}

Now we are in a position to present a second application of distance coordinates. We shall use some ideas of \cite{Ber00:subm} where the analogous result is proved in Riemannian setting.

\begin{thm}\label{thm:Beres}
A surjective regular submetry between reversible Finsler manifolds is differentiable.
\end{thm}

\begin{proof}
Let $(M,F)$ and $(N,\bar F)$ be reversible Finsler manifolds, and let $\varphi$ be a surjective regular submetry from $M$ to $N$. Choose a point $p\in M$ and a distance coordinate system $(\ir D,(\theta^i)\en)$ at $\varphi(p)$ such that the base points and $\ir D$ are contained in a totally normal neighbourhood of $\varphi(p)$. The functions $\theta^i$ are distance functions on a reversible Finsler manifold, therefore, by the previous lemma, they are regular submetries on $\ir D$. So the functions $\theta^i\circ\varphi$ are also regular submetries on $\varphi\inv(\ir D)$. If these functions are differentiable, then $\theta\circ\varphi$ is also differentiable, therefore $\varphi$ itself is differentiable. Consequently, we only have to show that regular submetries from a reversible Finsler manifold $(M,F)$ into $\RR$ are differentiable.

To do this, let $r$ be such a submetry, and let $q$ be a point in $M$. Select an open neighbourhood $D$ of $q$ with compact closure and a number $\delta\in\left]0,\delta_{\oper{cl}(D)}\right[$ such that the closure of $B_\varrho(q,\delta)$ is contained in $D$. Then for any point $p\in \oper{cl}(B_\varrho(q,\delta))$ we have $\delta_p>\delta$.

Consider the fibers $H^+:=r\inv(\{r(q)+\delta\})$ and $H^-:=r\inv(\{r(q)-\delta\})$. Since $r$ is a submetry, there is a point $b\in H^+$ and a point $a\in H^-$ such that $\varrho(q,b)=\varrho(q,a)=\delta$. Then $\delta_a>\delta$ and $\delta_b>\delta$, so $B_\varrho(a,\delta_a)\cap B_\varrho(b,\delta_b)$ is an open neighbourhood of $q$. Define the functions $f_a$ and $f_b$ on this set by $f_a(u):=r(a)+r_{a}(u)$ and $f_b(u):=r(b)-r_{b}(u)$. These functions have the following properties:
\begin{enumerate}[(i)]
\itemsep-3pt
\item $f_b\leq r\leq f_a$;
\item $f_b(q)=r(q)=f_a(q)$.
\end{enumerate}
Indeed, since $r$ is submetry,
\begin{align*}
r_{a}(u)=\varrho(a,u)\geq|r(u)-r(a)|\geq r(u)-r(a).
\end{align*}
Similarly, $r_{b}(u)\geq r(b)-r(u)$, and we obtain (i). Furthermore,
\begin{align*}
f_a(q)&=r(a)+r_{a}(q)\overset{a\in H^-}=r(q)-\delta+\varrho(q,a)=r(q),\\
f_b(q)&=r(b)-r_{b}(q)\overset{b\in H^+}=r(q)+\delta-\varrho(q,b)=r(q),
\end{align*}
so (ii) is also true. The function $f_a-f_b$ is differentiable on $B_\varrho(a,\delta_a)\cap B_\varrho(b,\delta_b)$, non-negative and vanishes at $q$, so it has a local minimum at that point. Therefore its differential vanishes at $q$, which implies that $(df_b)_q=(df_a)_q$.
Now let $\sigma$ be a differentiable curve in $M$ with $\sigma(0)=q$. Then, taking into account (i) and (ii), we find that
\[
\lim_{t\to0}\frac{r\circ\sigma(t)-r\circ \sigma(0)}t \geq\lim_{t\to0}\frac{f_b\circ\sigma(t)-f_b\circ \sigma(0)}t=(df_b)_q(\dot\sigma(0)),
\]
\[
\lim_{t\to0}\frac{r\circ\sigma(t)-r\circ \sigma(0)}t \leq\lim_{t\to0}\frac{f_a\circ\sigma(t)-f_a\circ \sigma(0)}t=(df_a)_q(\dot\sigma(0)),
\]
therefore $r$ is differentiable at $q$.
\end{proof}

\medskip

\noindent{\bf Bernadett Aradi}\\
{\it Institute of Mathematics,\\
MTA-DE Research Group ``Equations, Functions and Curves''\\
Hungarian Academy of Sciences and University of Debrecen\\
H-4010 Debrecen, P.O. Box 12, Hungary}\\
E-mail: \verb"bernadett.aradi@science.unideb.hu"\\
\medskip

\noindent{\bf D\'avid Csaba Kert\'esz}\\
{\it Institute of Mathematics, University of Debrecen\\
H-4010 Debrecen, P.O. Box 12, Hungary}\\
E-mail: \verb"kerteszd@science.unideb.hu"\\

\begin{thebibliography}{99}
\bibitem{BCS00:intr}
{D.~Bao, S.-S. Chern and Z.~Shen}, {\it An Introduction to Riemann-Finsler Geometry}, Springer-Verlag (New York, 2000).

\bibitem{Ber00:subm}
{V.~N.~Berestovskii and L.~Guijarro}, A Metric Characterization of Riemannian Submersions, {\it Ann. Global Anal. Geom.}, {\bf 18}(6) (2000), 577--588.

\bibitem{Bric65:diff}
{F.~Brickell}, On the differentiability of affine and projective transformations, {\it Proc. Amer. Math. Soc.}, {\bf 16} (1965), 567-–574.

\bibitem{BusMay48}
{H.~Busemann and W.~Mayer}, On the foundations of calculus of variations, {\it Trans. Amer. Math. Soc.}, {\bf 49} (1948), 173--198.

\bibitem{Deng12:hom}
{S.~Deng}, {\it Homogeneous Finsler Spaces}, Springer (2012).

\bibitem{DeHo02:isom}
{S.~Deng and Z.~Hou}, The group of isometries of a Finsler space, {\it Pacific J. Math.}, {\bf 207} (2002), 149--155.

\bibitem{GHV1}
{W.~Greub, S.~Halperin and R.~Vanstone}, {\it Connections, Curvature, and Cohomology}, Vol.~I, Academic Press  (New York, 1972).

\bibitem{Lang95:diff}
{S.~Lang}, {\it Differential and Riemannian Manifolds}, Springer-Verlag (New York, 1995).

\bibitem{LibMo}
{H.~Libing and X.~Mo}, Geodesics of Invariant Finsler Metrics on a Lie Group, {\tt http://www.math.pku.edu.cn:8000/var/preprint/649.pdf}.

\bibitem{Lyt}
{A.~Lytchak}, Differentiation in metric spaces, {\it St. Petersburg Math. J.}, {\bf 16} (2004), 1017–-1041.

\bibitem{ONeill}
{B.~O'Neill}, {\it Semi-Riemannian Geometry}, Academic Press (New York, 1983).

\bibitem{Pat05}
{M.~Patr\~ao}, Homotheties and isometries of metric spaces, {\it Matem\'atica Contempor\^anea}, {\bf 29} (2005), 79--97.

\bibitem{PSz}
{J.~P\'ek and J.~Szilasi}, Automorphisms of Ehresmann Connections, {\it Acta Math. Hungar.}, \textbf{123} (2009), 379--395.

\bibitem{Pet00:Riem}
{P.~Petersen}, {\it Riemannian Geometry}, Springer-Verlag (New York, 1998).

\bibitem{Shen01:Spra}
{Z.~Shen}, {\it Differential Geometry of Spray and Finsler Spaces}, Kluwer Academic Publishers (Dordrecht, 2001).

\bibitem{SzLK11}
{J.~Szilasi, R.~L. Lovas, and D.~\relax{Cs}. Kert\'esz}, Several Ways to a Berwald Manifold -- and Some Steps Beyond, {\it Extracta Math.}, {\bf 26} (2011),  89--130.

\bibitem{SzV}
{J.~Szilasi and Sz.~Vattam\'any}, On the Finsler-metrizabilities of spray manifolds, {\it Period. Math. Hungar.}, {\bf 44}(1) (2002), 81--100.

\bibitem{SziZ10:diss}
{Z.~Szilasi}, {\it On the Projective Theory of Sprays with Applications to Finsler Geometry}, PhD Thesis (Debrecen, 2010), {\tt http://arxiv.org/abs/0908.4384}.

\bibitem{Trab37:fund}
{R.~E. Traber}, A fundamental lemma on normal coordinates and its applications, {\it Q. J. Math.}, {\bf 8} (1937), 142--147.

\bibitem{Whit33:totnor}
{J.~H.~C.~Whitehead}, Convex regions in the geometry of paths -- addendum, {\it Q. J. Math.}, {\bf 4} (1933), 226--227.
\end{thebibliography}
\end{document}